\documentclass[reqno,11pt]{amsart}
\usepackage{amssymb,amsmath,amsthm,amstext,amsfonts}
\usepackage{graphicx}

 \makeatletter \@addtoreset{equation}{section} \makeatother

\theoremstyle{plain}
\newtheorem{theorem}{Theorem}[section]
\newtheorem{proposition}[theorem]{Proposition}
\newtheorem{lemma}[theorem]{Lemma}
\newtheorem{corollary}[theorem]{Corollary}
\theoremstyle{definition} \theoremstyle{remark}
\newtheorem{remark}[theorem]{Remark}
\newtheorem{definition}[theorem]{Definition}

\newcommand{\R}{{\mathbb R}}
\newcommand{\D}{{\mathbb D}}
 \newcommand{\Z}{{\mathbb Z}}
 \newcommand{\N}{{\mathbb N}}

  \newcommand{\cC}{{\mathcal C}}
  \newcommand{\cD}{{\mathcal D}}
  \newcommand{\cG}{{\mathcal G}}
  \newcommand{\cM}{{\mathcal M}}
  \newcommand{\cU}{{\mathcal U}}

 \def \fX {{\mathfrak X}}

\newcommand{\eps}{{\epsilon}}

\renewcommand{\epsilon}{{\varepsilon}}
 \newcommand{\Lip}{\operatorname{Lip}}
 \newcommand{\diver}{\operatorname{div}}
 \newcommand{\graph}{\operatorname{graph}}

\begin{document}

\title{Existence and smoothness of the stable foliation for sectional hyperbolic attractors}

\author{V. Ara\'ujo and I. Melbourne}

\address{Vitor Ara\'ujo,
 Departamento de Matem\'atica, Universidade Federal da Bahia\\
Av. Ademar de Barros s/n, 40170-110 Salvador, Brazil.}
\email{vitor.d.araujo@ufba.br,
 www.sd.mat.ufba.br/$\sim$vitor.d.araujo}

\address{Ian Melbourne,
Institute of Mathematics, University of Warwick, Coventry CV4 7AL, UK}
\email{i.melbourne@warwick.ac.uk}

\thanks{I.M. is partially supported by 
	by a European Advanced Grant StochExtHomog (ERC AdG 320977) and by CNPq
  (Brazil) through PVE grant number 313759/2014-6.  V.A.
  is partially supported by CNPq,
  PRONEX-Dyn.Syst. and FAPESB (Brazil).  We are grateful to Sheldon Newhouse for provocative and helpful questions about stable foliations.  
}

  \begin{abstract}
    We prove the existence of a contracting invariant topological foliation in a full neighborhood for partially hyperbolic attractors.  Under certain bunching conditions it can then be shown that this stable foliation is smooth.
Specialising to sectional hyperbolic attractors, we give a verifiable condition for bunching.
In particular, we show that the stable foliation for the classical Lorenz equation (and nearby vector fields) is better than $C^1$ which is crucial for recent results on exponential decay of correlations.  In fact the foliation is at least $C^{1.278}$.  
  \end{abstract}

 \date{December 21, 2016}

\maketitle

 \section{Introduction}  \label{sec:intro}

The goal of this paper is to establish existence and smoothness of the stable foliation for sectional hyperbolic flows.  In particular, we treat the case of the classical Lorenz equations~\cite{Lorenz63}
\begin{align*}
\dot x_1 &= 10(x_2 - x_1)
 \nonumber \\
\dot x_2 &= 28x_1 -x_2 -x_1x_3
\\
\dot x_3 &= x_1x_2 - \textstyle{\frac83}x_3
\end{align*}
showing that the stable foliation for the flow is at least $C^{1.278}$.
This regularity ($C^{1+\eps}$ for some $\eps>0$) is a crucial component of the analysis in~\cite{AMsub,AMV15} where we prove exponential decay of correlations for the Lorenz attractor.  
An immediate consequence of our result is that the stable foliation for the associated Poincar\'e map is also at least $C^{1.278}$.   
The results are robust in the sense that we obtain smoothness of the stable foliations and exponential decay of correlations for smooth vector fields that are sufficiently $C^1$-close to the classical one.

As far as we know, this is the first complete proof that the stable foliation for the classical Lorenz equations (or even for the Poincar\'e map) exists and is better than H\"older continuous.  
By~\cite{MPP04} and~\cite{Tucker}, the classical Lorenz attractor is a singular hyperbolic attractor.   A consequence is the existence of smooth stable leaves through each point of the attractor.  However, {\em a priori} it does not follow that these leaves form a topological foliation in a full neighborhood of the attractor; nor is there any information about smoothness of such a foliation.  
These issues are somewhat controversial, with various false claims in the literature.  Careful analyses (see for example~\cite{Homburg96,Robinson92,Rychlik90}) require additional conditions to establish smoothness and do not apply to the classical Lorenz attractor.   

Recently~\cite{AMV15} gave a verifiable criterion for smoothness of the stable foliation that is easily seen to hold for the classical Lorenz attractor.  However, the argument in~\cite{AMV15} presupposes that the stable leaves topologically foliate a full neighborhood of the attractor -- a fact that is folklore but for which apparently there is no proof available in the literature.

In this paper, we consider general partially hyperbolic attractors and give a complete proof of the existence of a topological foliation $\{W^s_x\}$ in a neighborhood of such attractors. The individual leaves $W^s_x$ are smoothly embedded stable manifolds.   
In general, the leaves need not vary smoothly, but under a bunching condition~\cite{HPS77} the foliation is smooth.
The argument in~\cite{AMV15} now applies, and we obtain existence and smoothness of the stable foliation for the classical Lorenz attractor.  Our results hold for the flow, and hence also for the Poincar\'e map.  This resolves an issue in~\cite[Section~2.4]{Tucker} where it is claimed that the stable foliation for the Poincar\'e map is smooth but no details are provided.

In addition, we extend the verifiable criterion of~\cite{AMV15} to the sectional hyperbolic situation, and we give a lower bound for the smoothness for the classical Lorenz attractor.    
The condition is verifiable in the sense that it depends only on the linearised vector field and the location of the attractor and its equilibria.

In Section~\ref{sec:Lorenzmodel}, we recall the notion of partially hyperbolic and sectional hyperbolic attractors.
Section~\ref{sec:cone} contains general facts about cone fields for partially hyperbolic attractors, as well as the crucial step that the stable bundle extends continuously to a contracting invariant bundle over a neighborhood of the attractor.
Section~\ref{sec:foliation} contains general results about the stable foliation of partially hyperbolic attractors.   In particular, the stable leaves define a
topological foliation of a neighborhood of the attractor and is smooth under a bunching condition.
In Section~\ref{sec:SD}, we specialise to sectional hyperbolic attractors.  Following~\cite{AMV15}, we give a verifiable condition for smoothness of the stable foliation and apply this to the classical Lorenz attractor.


\section{Sectional hyperbolic attractors}
\label{sec:Lorenzmodel}

In this section, we define what is understood as a sectional
hyperbolic attractor; see e.g. \cite{AraPac2010} for an
extended presentation of this theory,

Let $M$ be a compact Riemannian manifold and
${\fX}^r(M), \, r \geq 1,\,$ be the set of $C^r$ vector
fields defined on $M$.  Let $X_t$ denote the flow generated by $G\in\fX^r(M)$.  Given a compact invariant
set $\Lambda$ for $G\in {\fX}^r(M)$, we say that $\Lambda$
is \emph{isolated} if there exists an open set $U\supset
\Lambda$ such that
$$
\textstyle \Lambda =\bigcap_{t\in \R}X_t(U).
$$
If $U$ above can be chosen such that $X_t(U)\subset U$ for
$t>0$, then we say that $\Lambda$ is an \emph{attracting set}.

Given $x \in M$, we define $\omega_G(x)$ as the set of
accumulation points of the set $\{X_tx; t \geq
0\}$ and define $\alpha_G(x)=\omega_{-G}(x)$.
A subset $\Lambda \subset M$ is \emph{transitive} if it has
a full dense orbit, that is, there is $x\in \Lambda$ such that
$\omega_G(x)=\Lambda=\alpha_G(x)$.

\begin{definition}\label{def:attractor}
  An \emph{attractor} is a transitive attracting set, and a
  \emph{repeller} is an attractor for the reversed vector
  field $-G$.
\end{definition}

\begin{definition}
\label{def:PH}
Let $\Lambda$ be a compact invariant set for $G \in {\fX}^r(M)$.  We say that
$\Lambda$ is {\em partially hyperbolic} if the
tangent bundle over $\Lambda$ can be written as a continuous $DX_t$-invariant sum 
$$
T_\Lambda M=E^s\oplus E^{cu},
$$
where $d_s=\dim E^s\ge1$ and $d_{cu}=\dim E^{cu}\ge2$,
and there exist constants $C>0$, $\lambda\in(0,1)$ such that
for every $t > 0$ and every $x \in \Lambda$, we have
\begin{itemize}
\item uniform contraction along $E^s$:
\begin{equation}\label{eq:contrai}
\|DX_t \mid E^s_x\| \le C \lambda^t;
\end{equation}

\item domination of the splitting:
\begin{equation}\label{eq:domination}
\|DX_t \mid E^s_x\| \cdot
\|DX_{-t} \mid E^{cu}_{X_tx}\| \le C \lambda^t.
\end{equation}
\end{itemize}
We refer to $E^s$ as the stable bundle and to
$E^{cu}$ as the center-unstable bundle.
\end{definition}

\begin{remark} By~\cite[Theorem 1]{Goum07}, we may suppose
  without loss that $\|\cdot\|$ is an adapted metric so that
  $C=1$.
\end{remark}

\begin{definition} \label{def:VE}
The center-unstable bundle $E^{cu}$ is \emph{volume
  expanding} if there exists $K,\theta>0$ such that
$|\det(DX_t(x)\mid E^{cu}_x)|\geq K \, e^{\theta t}$ for all
$x\in \Lambda$ and $t\geq 0$.
More generally, $E^{cu}$ is {\em sectional expanding} if 
for every two-dimensional subspace $P_x\subset E^{cu}_x$,
\begin{align*} 
|\det(DX_t(x)\mid P_x )| \ge K  e^{\theta t}\quad\text{for all 
$x \in \Lambda$, $t\ge0$}. 
\end{align*}
\end{definition}

If $\sigma\in M$ and $G(\sigma)=0$, then $\sigma$ is called an {\em equilibrium}.
An invariant set is \emph{nontrivial} if it is neither a periodic orbit nor an
equilibrium. 

\begin{definition} 
\label{def:singularset}
Let $\Lambda$ be a compact nontrivial invariant set for $G \in
{\fX}^r(M)$.  We say that $\Lambda$ is a
\emph{sectional hyperbolic set} if all the equilibria
in $\Lambda$ are hyperbolic, and $\Lambda$ is partially
hyperbolic with sectional expanding center-unstable bundle. A
sectional hyperbolic set which is also an attractor is
called a {\em {sectional hyperbolic attractor}}.

In the special case when $E^{cu}$ is volume expanding, $\Lambda$ is called a {\em singular} hyperbolic set/attractor.
\end{definition}

An isolated set $\Lambda=\Lambda_G$ for a $C^1$ vector field $G$ is
\emph{robustly transitive} if there is an open set $U\supset\Lambda$ such that
$\Lambda_{\tilde G}=\bigcap_{t\in\R}\tilde X_t(U)$
is transitive and nontrivial for any vector field $\tilde G$ $C^1$-close to $G$.

\begin{definition}
 \label{def:lorenz-like}
 An equilibrium $\sigma$ for a $3$-dimensional vector field $G$ is
 {\em Lorenz-like} if the eigenvalues $\lambda_j$, $1\leq j \leq
 3$, of $DG(\sigma)$ are real and satisfy
  $\lambda_1<\lambda_2<0<-\lambda_2<\lambda_3$.
\end{definition}

For $3$-dimensional vector fields,
 Morales, Pacifico, and Pujals proved in~\cite{MPP04} that
 any robustly transitive invariant set $\Lambda$
 containing an equilibrium is a
 singular hyperbolic attractor or repeller. Moreover,
 every equilibrium in $\Lambda$ is Lorenz-like for $G$ or $-G$, and $\Lambda$
 is proper, i.e., $\Lambda\neq M$.

Tucker~\cite{Tucker} gave a computer-assisted proof that the classical Lorenz attractor~\cite{Lorenz63} is a robustly transitive invariant set containing an equilibrium.   It then follows from~\cite{MPP04} that the classical Lorenz attractor is singular hyperbolic.

\section{Cone fields and the stable bundle for partially hyperbolic attractors}
\label{sec:cone}

In this section, we analyse stable and center-unstable cone fields in a neighborhood of a partially hyperbolic attractor $\Lambda$, and 
we show that the stable bundle $E^s$ extends to a 
continuous $DX_t$-invariant contracting bundle over a neighborhood of $\Lambda$.

Throughout, $\Lambda$ is a partially hyperbolic attractor for a vector field
$G\in {\fX}^r(M)$, $r\ge1$, with invariant splitting $T_\Lambda M =
E^s\oplus E^{cu}$ and contraction rate $\lambda\in(0,1)$.
Sectional expansion is not assumed.
Write $d=\dim M=d_s+d_{cu}$.

\subsection{Cone fields in a neighborhood of $\Lambda$}
\label{sec:conefields}

Let $U_0\subset M$ be a forward invariant neighborhood of
$\Lambda$ such that $\bigcap_{t\geq 0}X_t(U_0)=\Lambda$.
Choose a continuous (not necessarily invariant) extension
$T_{U_0}M=E^s\oplus E^{cu}$ of the splitting $T_\Lambda
M=E^s\oplus E^{cu}$.  Given $x\in U_0$ and $a>0$ we define
the cone fields
\begin{align*}
  \cC^s_x(a) 
  & =
  \{v=v^s+v^{cu}\in E^s_x\oplus E^{cu}_x:
  \|v^{cu}\|\le a\|v^s\|\}, 
  \\
  \cC^{cu}_x(a) 
  & =
  \{v=v^s+v^{cu}\in E^s_x\oplus E^{cu}_x: 
  \|v^s\|\le a\|v^{cu}\|\}.
\end{align*}

\begin{proposition} \label{p:conefield} 
Fix $T$ so that $\lambda^T=1/150$.   For any $a\in(0,\frac14]$ there
exists 
a positively invariant neighborhood $U_0$ of $\Lambda$,  
such that for all $x\in U_0$ the following hold:
\begin{itemize}
	\item[(a)] backward invariance of
  stable cones and forward invariance of center-unstable cones:
\begin{align}
\label{eq:cone-s}
& DX_{-t}\big(\cC^s_{X_tx}(b)\big) \subset \cC^s_x(b) ,
\\  \label{eq:cone-u}
& DX_t\big(\cC^{cu}_x(b)\big) \subset \cC^{cu}_{X_tx}(b),
\end{align}
for all $b\ge a$, $t\ge T$.
\item[(b)] backward expansion of stable cones and domination: 
	there exist constants \mbox{$c>0$}, $\tilde\lambda\in(0,1)$, such that
for all $t>0$,
\begin{align}
\label{eq:cone-contract}
& \|DX_{-t}(X_tx)v\|\ge  c\tilde\lambda^{-t}\|v\|
 \quad\text{for all $v\in \cC^s_{X_tx}(a)$,}
 \\
\label{eq:domincone}
& \frac{\|DX_t(x) v\|}{\|v\|}\ge c\tilde\lambda^{-t}\frac{\|DX_t(x)u\|}{\|u\|}
 \quad\text{for all nonzero 
 $v\in \cC^{cu}_x(a)$, $u\in DX_{-t}(\cC^s_{X_tx}(a)).$}
\end{align}
\end{itemize}
\end{proposition}

\begin{proof}
  If $v$ lies in $T_xM$ where $x\in U_0$, then we write $v=v^s+v^{cu}\in E^s_x\oplus E^{cu}_x$.
If $v\in\cC^*_x(a)$, then $(1-a)\|v^*\|\le\|v\|\le(1+a)\|v^*\|$
where throughout $*\in\{s,cu\}$.

  For $x\in\Lambda$, it follows from invariance of the splitting $E^s\oplus E^{cu}$ that $(DX_t(x)v)^*=DX_t(x)v^*$ 
  for all $v\in T_xM$ and $t\in\R$.
  
We fix the neighborhood $U_0$ as follows.
For each $x\in\Lambda$, we choose a neighborhood $U_x\subset M$ of $x$ such that
$U_x$ is diffeomorphic to $\R^d$ where $d=\dim M$.
Then $T_{U_x}M$ is identified with $U_x\times\R^d$.
Given $y_1,y_2\in U_x$, a vector $v\in \R^d$ corresponds to vectors $v_{y_j}\in T_{y_j}M$ via this identification.  Using the smoothness of the flow, we can choose $U_x$ so small that 
	$\|DX_t(y_1)v_{y_1}\|\le 2\|DX_t(y_2)v_{y_2}\|$
for all $x\in\Lambda$, $y_1,y_2\in U_x$, $v\in \R^d$, $t\in[-T,T]$.
Using moreover the continuity of the splitting $E^s\oplus E^{cu}$, for $a>0$ fixed we can ensure for all $b\ge a/8$, $t\in[-T,T]$,
that if $DX_t(y_1) v_{y_1}\in \cC^*_{X_ty_1}(b)$, then
$DX_t(y_2) v_{y_2}\in \cC^*_{X_ty_2}(2b)$.

We now fix $U_0$ to be a positively invariant neighborhood of $\Lambda$ contained in 
$\bigcup_{x\in\Lambda}U_x$.   
By construction, for every $y\in U_0$, there exists $x\in\Lambda$ such that
\begin{itemize}
\item[(i)] 
	$DX_t(x) v_x\in \cC^*_{X_tx}(b)$ implies that
	$DX_t(y) v_{y}\in \cC^*_{X_ty}(2b)$, 
\item[(ii)] 
	$DX_t(y) v_y\in \cC^*_{X_ty}(b)$ implies that
	$DX_t(x) v_x\in \cC^*_{X_tx}(2b)$, and
\item[(iii)] 	$\frac12\|DX_t(x)v_x\|\le \|DX_t(y)v_y\|\le 2\|DX_t(x)v_x\|$,
\end{itemize}
for all $v\in\R^d$, $b\ge a/8$, $t\in[-T,T]$.

We now proceed with the proof of part~(a).
  By~\eqref{eq:domination},
  \begin{align*}
  \|(DX_t(x)v)^s\| & = \|DX_t(x)v^s\|
	  \le \|DX_t|E^s_x\|\|v^s\|\le \lambda^t \|DX_{-t}|E^{cu}_{X_tx}\|^{-1}\|v^s\| \\ &
	  = \lambda^t \|(DX_t|E^{cu}_x)^{-1}\|^{-1}\|v^s\|
	  \le  \lambda^t \|(DX_t(x)v)^{cu}\|\|v^{cu}\|^{-1}\|v^s\|,
\end{align*}
  for all $x\in\Lambda$, $v\in T_xM$, $t\ge 0$.
In particular, 
$DX_t(\cC^{cu}_x(b))\subset\cC^{cu}_{X_tx}(b\lambda^t)$ for all $x\in\Lambda$, $b>0$, $t\ge0$.

Now let $y\in U_0$, $b\ge a$, $v\in\cC^{cu}_y(b)$.
We can pass to a nearby point $x\in\Lambda$ with corresponding
vector $v_x\in\cC^{cu}_x(2b)$ by~(ii).
Then $DX_t(x)v_x\in\cC^{cu}_{X_tx}(2b\lambda^t)$ for all $t\ge0$.
In particular, since $\lambda^T=1/150\le1/16$,
\[
	DX_T(x)v_x\in\cC^{cu}_{X_Tx}(b/8) \quad\text{and}\quad
	DX_t(x)v_x\in\cC^{cu}_{X_tx}(2b)\,\text{for all $t\ge0$}.
\]
By~(i),
\begin{align} \label{eq:T}
	DX_T(\cC^{cu}_y(b))\subset\cC^{cu}_{X_Ty}(b/4) \subset \cC^{cu}_{X_Ty}(b) 
	\quad\text{and}\quad
	DX_r(\cC^{cu}_y(b))\subset\cC^{cu}_{X_ry}(4b),
\end{align} 
for all $r\in[0,T]$, $y\in U_0$.

By positive invariance of $U_0$, it follows inductively from~\eqref{eq:T} that
$DX_{kT}(\cC^{cu}_y(b))\subset \cC^{cu}_{X_{kT}y}(b/4)
\subset \cC^{cu}_{X_{kT}y}(b)$ for all $y\in U_0$,
$k\in\N$.

For general $t\ge T$, write $t=kT+r$ where $k\ge1$ and $r\in[0,T)$.
	Again using positive invariance of $U_0$ together with~\eqref{eq:T},
\[
DX_t(\cC^{cu}_y(b))=DX_{kT}\cdot DX_r(\cC^{cu}_y(b))
\subset DX_{kT}(\cC^{cu}_{X_ry}(4b))\subset \cC^{cu}_{X_ty}(b).
\]
This completes the proof of~\eqref{eq:cone-u}.

The proof of~\eqref{eq:cone-s} is similar, so we only sketch the details.
Using~\eqref{eq:domination} as before, we obtain that
$DX_{-t}(\cC^s_{X_tx}(b))\subset\cC^s_x(b\lambda^t)$ for all $x\in\Lambda$, $b>0$, $t\ge0$.
Let $y\in U_0$, $b\ge a$, $v\in \cC_{X_ty}^s(b)$ where $t\ge0$,
and pass to a nearby point $x_t\in\Lambda$ such that
$v_{X_tx_t}\in \cC^s_{X_tx_t}(2b)$.
(The only difference here is the dependence of $x_t$ on $t$.)
As before, we obtain that
\[
	DX_{-T}(X_Tx_T)v_{X_Tx_t}\in\cC^s_{x_T}(b/8) \quad\text{and}\quad
	DX_{-t}(X_tx_t)v_{X_tx_t}\in\cC^s_{x_t}(2b)\,\text{for all $t\ge0$},
\]
from which it follows that
\begin{align*} 
	DX_{-T}(\cC^s_{X_Ty}(b))\subset\cC^s_y(b/4) \subset \cC^s_y(b) 
	\quad\text{and}\quad
	DX_{-r}(\cC^s_{X_ry}(b))\subset\cC^s_y(4b),
\end{align*} 
for all $r\in[0,T]$, $y\in U_0$.
The last formulas are the direct analogy to those in~\eqref{eq:T}, and the 
remainder of the proof of~\eqref{eq:cone-s}
is identical to the proof of~\eqref{eq:cone-u}.

Next we turn to part (b).  The choices of $T$ and $U_0$ are unchanged.
Recall that $a\in(0,\frac14]$ is fixed.
First we prove~\eqref{eq:cone-contract}.
Suppose that $x\in \Lambda$ and $v\in \cC^s_{X_Tx}(2a)$.
By~\eqref{eq:cone-s}, $DX_{-T}(X_Tx)v\in \cC^s_x(2a)$, so using~\eqref{eq:contrai},
\begin{align*}
	\|DX_{-T}(X_Tx)v\| & \ge (1-2a)\|(DX_{-T}(X_Tx)v)^s\|
	= (1-2a)\|(DX_T(x))^{-1}v^s\|
	\\ & \ge (1-2a)\lambda^{-T}\|v^s\|\ge (1+2a)^{-1}(1-2a)\lambda^{-T}\|v\|
\ge 50\|v\| \ge 8 \|v\|.
\end{align*}

Now let $y\in U_0$, $v\in\cC^s_{X_Ty}(a)$.   As in part~(a), we can pass to
a nearby point $x\in \Lambda$ with corresponding vector $v_x\in\cC^s_{X_Tx}(2a)$ and so 
$\|DX_{-T}(X_Tx)v_x\|   \ge 8\|v_x\|$.
Using~(iii) together with positive invariance of $U_0$, we have that
$\|DX_{-T}(X_Ty)v\|   \ge 2 \|v\|$ for all $v\in \cC^s_{X_Ty}(a)$.

By positive invariance of $U_0$ and~\eqref{eq:cone-s},
it follows inductively that
\begin{align} \label{eq:k}
	\|DX_{-kT}(X_{kT}y)v\| \ge 2^k\|v\| \quad \text{for $y\in U_0$,
	$v\in \cC^s_{X_{kT}y}(a)$, $k\ge0$.}
\end{align}
Finally, we consider the case of general $t=kT+r$ where $k\in\N$, $r\in[0,T)$.
Let $v\in\cC^s_{X_ty}(a)$.
Then $DX_{-t}(X_ty)v=DX_{-r}(X_{r}y)\,DX_{-kT}(X_ty)v$ 
so it follows from positive invariance and~\eqref{eq:k} that 
\[
	\|DX_{-t}(X_ty)v\|\ge c\|DX_{-kT}(X_{kT}(X_ry))v\|\ge c2^k  \|v\|,
\]
where $c=\inf_{r\in[0,T]}\inf_{y\in U_0}\inf_{v\in T_yM,\,v\neq0}
\|DX_{-r}(y)v\|/\|v\|>0$.
This completes the proof of~\eqref{eq:cone-contract}.

To prove~\eqref{eq:domincone}, we start from~\eqref{eq:domination} so for $x\in \Lambda$, 
$u,v\in T_xM$,
\[
	\frac{\|DX_T(x)u^s\|}{\|u^s\|}\le \|DX_T|E^s_x\|\le \lambda^T\|(DX_T|E^{cu}_x)^{-1}\|^{-1}\le \lambda^T\frac{\|DX_T(x)v^{cu}\|}{\|v^{cu}\|}.
\]
	Let $u\in DX_{-T}(\cC^s_{X_Tx}(2a))$, 
	$v\in \cC^{cu}_x(2a)$.  By~\eqref{eq:cone-s} and~\eqref{eq:cone-u},
	\[
		\frac{\|DX_T(x)v^{cu}\|}{\|v^{cu}\|}\le 
		\frac{(1+2a)\|DX_T(x)v\|}{(1-2a)\|v\|}, 
		\quad\text{and}\quad
\frac{\|DX_T(x)u\|}{\|u\|}\le 
	 \frac{(1+2a)\|DX_T(x)u^s\|}{(1-2a)\|u^s\|},
\]
and so
\[
	\frac{\|DX_T(x)u\|}{\|u\|}\le 9
	\lambda^T\frac{\|DX_T(x)v\|}{\|v\|}
	\le \frac{3}{50}\frac{\|DX_T(x)v\|}{\|v\|}
\]
for all $v\in \cC^{cu}_x(2a)$, $u\in DX_{-T}(\cC^s_{X_Tx}(2a))$.
Using~(iii), it follows that
\[
	\frac{\|DX_T(y)u\|}{\|u\|}\le 
	\frac{24}{25}\frac{\|DX_T(y)v\|}{\|v\|}
\]
for all $v\in \cC^{cu}_y(a)$, $u\in DX_{-T}(\cC^s_{X_Ty}(a))$.
For general $t\ge0$, we write $t=kT+r$, $k\ge0$, $r\in[0,T)$ and proceed as in the proof of~\eqref{eq:cone-contract}.
\end{proof}

\subsection{Stable bundle over a neighborhood of $\Lambda$}
\label{sec:Es}

Whereas the original splitting $T_\Lambda M=E^s\oplus
E^{cu}$ is $DX_t$-invariant, in general the extension
$E^{cu}$ of the center-unstable bundle cannot be
assumed to be invariant. 
However the extension $E^s$ of the stable bundle may be chosen to
be $DX_t$-invariant:

\begin{proposition} \label{p:hatEs}
The continuous bundle $E^s$ over $U_0$ can be chosen to be
	$DX_t$-invariant and uniformly contracting:
	$\|DX_t \mid E^s_x\|\le c^{-1}\tilde\lambda^t$ for all $t\ge0$, $x\in U_0$,
where $c>0$, $\tilde\lambda\in(0,1)$ are the constants in Proposition~\ref{p:conefield}.
\end{proposition}

\begin{proof}
  We begin with the original choice of continuous splitting
  $T_{U_0}M=E^s\oplus E^{cu}$.  Let $a\in(0,\frac14]$ and choose
  $T$ and $U_0$ as in Proposition~\ref{p:conefield}.
  For $x\in U_0$, define
  \begin{align*}
\textstyle    F_x=\bigcap_{t\ge0} DX_{-t}\big(\cC^s_{X_tx}(a)\big).
  \end{align*}
We show that $\{F_x\}$ is the desired stable bundle.  That is, we show that
for all $t\ge0$,
\begin{itemize}
\item[(i)] $x\mapsto F_x$ is a continuous map from $U_0$ to 
	the Grassmannian bundle $\cG=\{\cG_x,\,x\in U_0\}$ where
	$\cG_x$ is the space of $d_s$-dimensional subspaces of $T_xM$,
 \item[(ii)] $F_x=E^s_x$ for $x\in\Lambda$,
 \item[(iii)] $\{F_x,\,x\in U_0\}$ is $DX_t$-invariant and uniformly contracting.
\end{itemize}

Now $\{DX_{-t}(\cC^s_{X_tx}(a)),\,t\ge0\}$ 
  is a nested family of closed cones, and by~\eqref{eq:cone-s} the cones are contained in $\cC^s_x(a)$ for $t\ge T$.  In particular,
  $F_x\subset \cC^s_x(a)$.
  
  We can also regard $\{DX_{-t}(\cC^s_{X_tx}(a)),\,t\ge0\}$ 
  as a nested family of closed subsets of $\cG_x$,
  so $F_x$ is a closed subset of $\cG_x$.
  By compactness of $\cG_x$,
  the elements $DX_{-t}E^s_{X_tx}\in\cG_x$ have a convergent subsequence
  $DX_{-t_n}E^s_{X_{t_n}x}$ with limit $\tilde F_x\in\cG_x$.
  Since $DX_{-t}E^s_{X_tx}\in
  DX_{-t}(\cC^s_{X_tx}(a))$ and $F_x$ is closed, it follows that
  $\tilde F_x\in F_x$.

To summarise, we have shown that there exists a  $d_s$-dimensional subspace
$\tilde F_x$ such that $\tilde F_x\subset F_x$ and $\tilde F_x=\lim_{n\to\infty}DX_{-t_n}E^s_{X_{t_n}x}$ (in $\cG_x$).  Without loss we may suppose that $t_n\ge T$ for all $n$.

  Next we show that $F_x=\tilde F_x$.
  Choose vectors $u_n\in E^s_{X_{t_n}x}$ such that
  \mbox{$\|DX_{-t_n}(X_{t_n}x)u_n\|=1$}.
	  Suppose for contradiction that $F_x\neq\tilde F_x$.
	  Then $F_x$ is a nontrivial cone containing $\tilde F_x$, and so there exists $v\in E^{cu}_x$ nonzero such that 
	  $w_n=DX_{-t_n}(X_{t_n}x)u_n+v\in F_x$ for $n$ sufficiently large.
  It follows from the definition of $F_x$ that
  $DX_{t_n}(x)w_n=u_n+DX_{t_n}(x)v\in \cC^s_{X_{t_n}x}(a)$.
Hence
\begin{align} \label{eq:w}
	\|(DX_{t_n}(x)v)^{cu}\|\le a\|u_n+(DX_{t_n}(x)v)^s\|.
\end{align}

Since $v\in E^{cu}_x$, it follows from~\eqref{eq:cone-u} that
$DX_{t_n}(x)v\in \cC^{cu}_x(a)$ and hence
$\|(DX_{t_n}(x)v)^s\|\le a\|(DX_{t_n}(x)v)^{cu}\|$ and
$\|DX_{t_n}(x)v\|\le (1+a) \|(DX_{t_n}(x)v)^{cu}\|$.
Substituting into~\eqref{eq:w} yields
$(1-a^2)\|(DX_{t_n}(x)v)^{cu}\|\le a\|u_n\|$ and then
\[
	\|DX_{t_n}(x)v\|\le (1+a)(1-a^2)^{-1}a\|u_n\|.
\]

On the other hand, $u_n\in E^s_{X_{t_n}x}$, $v\in E^{cu}_x$, so by~\eqref{eq:domincone},
\[
	\frac{\|DX_{t_n}(x)v\|}{\|v\|}\ge 
	c\tilde\lambda^{-t_n}\frac{\|u_n\|}{\|DX_{-t_n}(X_{t_n}x)u_n\|}
	=c\tilde\lambda^{-t_n}\|u_n\|.
\]
Letting $n\to\infty$ yields the desired contradiction,
  and so $F_x$ and $\tilde F_x$ coincide.
  In particular, $F_x\in\cG_x$ for all $x\in U_0$.

To prove continuity of the map $x\mapsto F_x$, fix $x\in U_0$ and let
$\cU\subset\cG$ be a neighborhood of $F_x$.
There exists $t_0\ge0$ such that
    $\bigcap_{t\le t_0} DX_{-t}(\cC^s_{X_tx}(a))\subset\cU$.
    By smoothness of the flow,
    $F_y\subset\bigcap_{t\le t_0} DX_{-t}(\cC^s_{X_ty}(a))\subset\cU$ for $y$ sufficiently close to $x$.
    This completes the proof of (i).
    
  It is immediate from invariance of the bundle $E^s|\Lambda$ that
  $E^s_x\subset F_x$ for all $x\in\Lambda$.   Hence
  $E^s_x=F_x$ for all $x\in\Lambda$ establishing (ii).

  For $r\ge0$,
  \begin{align*}
	  DX_rF_x & \textstyle =\bigcap_{t\ge0}DX_{r-t}(\cC^s_{X_{t-r}(X_rx)}(a))
	  =\bigcap_{t\ge r}DX_{r-t}(\cC^s_{X_{t-r}(X_rx)}(a))
	  \\ & \textstyle =\bigcap_{t\ge0}DX_{-t}(\cC^s_{X_t(X_rx)}(a))
	= F_{X_rx},
\end{align*}
so the bundle $\{F_x\}$ is $DX_t$-invariant.
Finally,  if $v\in F_x$, $t\ge0$, then $DX_t(x)v\in\cC^s_{X_tx}(a)$ so
by~\eqref{eq:cone-contract}, $\|v\|\ge c\tilde\lambda^{-t}\|DX_t(x)v\|$.
Hence 
$\|DX_t \mid F_x\|\le c^{-1}\tilde\lambda^t$ so (iii) holds.
\end{proof}

From now on, we suppose that the continuous extension
$T_{U_0}M=E^s\oplus E^{cu}$ of $T_\Lambda M=E^s\oplus
E^{cu}$ is chosen so that $E^s$ is invariant and uniformly
contracted.

\begin{remark}
  \label{rmk:dominationenough}
In the definition of partial hyperbolicity, we assumed uniform contraction along $E^s$ and dominated splitting.  However the uniform contraction assumption~\eqref{eq:contrai} was used only to ensure that the extended stable bundle is uniformly contracting.

For attracting sets satisfying just the dominated splitting assumption~\eqref{eq:domination}, it still follows from the arguments above that the bundle $E^s$ over $\Lambda$ extends 
to a continuous invariant bundle over a neighborhood of $\Lambda$.
%
\end{remark}

\begin{remark} \label{rmk-PH}
Let $r\ge1$.  A compact invariant set for a $C^r$ diffeomorphism $f:M\to M$
is {\em partially
  hyperbolic} if there is a continuous $Df$-invariant splitting
$T_\Lambda M=E^s\oplus E^{cu}$
where $\dim E^s\ge1$ and $\dim E^{cu}\ge1$,
and there exist constants $C>0$, $\lambda\in(0,1)$ such that
for every $n \ge 1$ and every $x \in \Lambda$, we have
\begin{align*}
\|Df^n \mid E^s_x\| \le C \lambda^n;\quad
\|Df^n \mid E^s_x\| \cdot
\|Df^{-n} \mid E^{cu}_{f^nx}\| \le C \lambda^n.
\end{align*}
It is easily seen that the results in this section go through for partially
hyperbolic attractors for diffeomorphisms, with $f$ playing the role of $X_1$ and $X_T$ replaced by a high enough iterate of $f$.
\end{remark}

\section{The stable foliation for partially hyperbolic attractors}
\label{sec:foliation}

In this section, we discuss the existence and regularity properties of the stable foliation associated with a partially hyperbolic attractor $\Lambda$ satisfying the conditions in Definition~\ref{def:PH}.   Sectional expansion is not assumed.
Again we focus on partially hyperbolic attractors for flows, but the situation for diffeomorphisms is the same (cf.\ Remark~\ref{rmk-PH}).

In Subsection~\ref{sec:Ws}, we prove that the stable bundle $E^s$ integrates to a contracting invariant topological foliation of a neighborhood of $\Lambda$ with smooth leaves.
In Subsection~\ref{sec:reg}, we obtain smoothness of the foliation under a suitable bunching condition.

\begin{remark}
The results in this section follow entirely from standard arguments.   However the proof that the extended stable bundle $E^s$ in Section~\ref{sec:Es} integrates to a topological foliation is complicated by the noninvariance of the complementary bundle $E^{cu}$.
Since we have been unable to find a formulation in the literature that does not assume invariance of both $E^s$ and $E^{cu}$, we present below the details of the standard arguments suitably modified.
\end{remark}

\subsection{Stable foliation in a neighborhood of $\Lambda$}
\label{sec:Ws}

Let $\D^k$ denote the $k$-dimensional open unit disk and let
$\mathrm{Emb}^r(\D^k,M)$ denote the set of $C^r$ embeddings $\phi:\D^k\to M$ endowed with the $C^r$ distance.

\begin{theorem}\label{th:Ws}
	There exists a positively invariant neighborhood $U_0$
	of $\Lambda$, and a constant $\nu\in(0,1)$, such
	that the following are true.

\vspace{1ex}
 \noindent(a)
For every point $x \in U_0$ there is a $C^r$ embedded $d_s$-dimensional disk
  $W^s_x\subset M$, with $x\in W^s_x$, such that
\begin{enumerate}
	\item $T_xW^s_x=E^s_x$.
\item $X_t(W^s_x)\subset W^s_{X_tx}$ for all $t\ge0$.
\item $d(X_tx,X_ty)\le \nu^t d(x,y)$ for all $y\in W^s_x$, $t\ge0$.
\end{enumerate}

\vspace{1ex}
\noindent(b) The disks $W^s_x$ depend continuously on $x$ in the $C^0$ topology: there is a continuous map $\gamma:U_0\to {\rm Emb}^0(\D^{d_s},M)$ such that
$\gamma(x)(0)=x$ and $\gamma(x)(\D^{d_s})=W^s_x$.

\vspace{1ex}
\noindent(c) The family of disks $\{W^s_x:x\in U_0\}$ defines a topological foliation of $U_0$.
\end{theorem}

To prove Theorem~\ref{th:Ws},
we begin by following the exposition in~\cite[Section 6.4(b)]{KH95}.  

Let $T>0$, $c>0$, $\tilde\lambda\in(0,1)$ be the constants in Propositions~\ref{p:conefield} and~\ref{p:hatEs}.
Increase $T>0$ if necessary so that $\hat\lambda=c^{-1}\tilde\lambda^T\in(0,1)$.   Define the diffeomorphism
  $f=X_T:U_0\to U_0$.

For each $x\in U_0$, we consider the exponential map $\exp_x:T_xM\to M$.
This transforms a small enough neighborhood of $0$ diffeomorphically onto a neighborhood of $x$, and $D\exp_x(0)=I$.

Choose orthonormal bases on $\R^{d_s}$, $\R^{d_{cu}}$.
Also for each $x\in U_0$, choose orthonormal bases on $E^s_x$ and $E^{cu}_x$.
Let $P^s_x:\R^{d_s}\to E^s_x$, 
$P^{cu}_x:\R^{d_{cu}}\to E^{cu}_x$ be the corresponding isometric isomorphisms. 
The splitting $E^s\oplus E^{cu}$ is continuous so we can arrange that
$x\mapsto P^s_x$ and $x\mapsto P^{cu}_x$ are continuous families of isomorphisms.

Define $P_{x,n}=P^s_{f^nx}+Df^n(x)P^{cu}_x:\R^d\to T_{f^nx}M$.
Note that $x\mapsto P_{x,n}$ is a continuous family of isomorphisms for each $n$.
In general
$P_{x,n}$ is not an isometric isomorphism since $Df^nE^{cu}_x$ is not necessarily orthogonal to $E^s_{f^nx}$.
However, it follows from~\eqref{eq:cone-u} that $Df^nE^{cu}_x\subset\cC^{cu}_{f^nx}(a)$ for some $a\in(0,\frac14]$, and so the angle between 
the subspaces $E^s_{f^nx}$ and $Df^nE^{cu}_x$ is bounded away from zero.
Hence there is a constant $C_1\ge1$ such that
$C_1^{-1}\le \|P_{x,n}\|\le C_1$ for all $x\in U_0$, $n\ge0$.

Next, $Q_{x,n}=\exp_{f^nx}\circ P_{x,n}:\R^d\to M$ maps a neighborhood of $0$ in $\R^d$ diffeomorphically onto a neighborhood of $f^nx$.
Again, $x\mapsto Q_{x,n}$ is a continuous family of diffeomorphisms for each $n$.

Let $D_\rho\subset\R^d$ denote the $\rho$-neighborhood of $0$.
Using boundedness of $\|P_n\|$ and compactness of $\Lambda$, and shrinking $U_0$ if necessary, we can choose $\rho>0$ so that $Q_{x,n}:D_\rho\to M$ is a diffeomorphism onto its range for all $n$.
Moreover, there is a constant $C_2\ge1$ such that
\[
C_2^{-1}\|p\|\le d(f^nx,Q_{x,n}(p))\le C_2\|p\|\;
\text{for all $x\in U_0$, $n\ge0$, $p\in D_\rho$.}
\]

Now define the family of maps $f_{x,n}=Q_{x,n+1}^{-1}\circ f\circ Q_{x,n}:D_\rho\to\R^d$.  By construction, $Df_{x,n}(0)$ is identified with $Df(f^nx)$.
Also, the maps $f_{x,n}$ are uniformly  $C^r$ close to
$Df_{x,n}(0)$ on $D_\rho$.  
Hence for any $\delta>0$ there exists $\rho>0$ and a family of (surjective) $C^r$ diffeomorphisms $g_{x,n}:\R^d\to\R^d$, $n\ge0$, such that $\|g_{x,n}-Df_{x,n}(0)\|_{C^1}<\delta$ and $g_{x,n}=f_{x,n}$ on $D_\rho$.  (See for example~\cite[Lemma 6.2.7]{KH95}.)


\begin{proposition} \label{p:Dg}
	For all $n\ge0$,
\[
\|Dg_{x,n}(0)\mid \R^{d_s}\|\le \hat\lambda, \quad
\|Dg_{x,n}(0)\mid \R^{d_s}\|\cdot\|Dg_{x,n}(0)^{-1}\mid\R^{d_{cu}}\| \le
\hat\lambda.
\]
\end{proposition}

\begin{proof}
Choose $a$ as in Proposition~\ref{p:conefield}.
By construction, $Dg_{x,n}(0)=Df_{x,n}(0)$ is identified with $Df(f^nx)$ and moreover,
\begin{align*}
	& \|Dg_{x,n}(0)\mid \R^{d_s}\|=\|Df\mid E^s_{f^nx}\|=\|DX_T\mid DX_{-T}E^s_{X_Tf^nx}\|,  \\
	& \|Dg_{x,n}(0)^{-1}\mid \R^{d_{cu}}\|
=\|Df^{-1}\mid Df^{n+1}E^{cu}_x\|
\le \|DX_{-T}\mid DX_T(\cC^{cu}_{f^nx}(a))\|,
\end{align*}
where we have used invariance of $E^s$ and forward invariance of $\cC^{cu}(a)$.
The second estimate follows from~\eqref{eq:domincone}.
The first estimate is immediate from Proposition~\ref{p:hatEs}.
\end{proof}

  We require a slightly modified version of the Hadamard-Perron Invariant Manifold Theorem from \cite[Theorem 6.2.8, pp 242-257]{KH95}:  

  \begin{lemma}\label{lem:HP}
	  Let $r\ge1$.
	  Fix $\lambda_{min}>0$, $\sigma\in(0,1)$.
    Then there exists $\gamma$, $\delta>0$ arbitrarily small so that
the following holds:

    For each $n$ let $g_n:\R^d\to\R^d$ be a $C^r$
    diffeomorphism such that
    \begin{align} \label{eq:gn}
      g_n(u,v)=(A_nu+\alpha_n(u,v), B_nv+\beta_n(u,v)),
      \quad
      (u,v)\in\R^{d_s}\oplus\R^{d_{cu}},
    \end{align}
    for linear maps $A_n:\R^{d_s}\to\R^{d_s},
    B_n:\R^{d_{cu}}\to\R^{d_{cu}}$ and $C^r$ maps
    $\alpha_n:\R^d\to\R^{d_s}, \beta_n:\R^d\to\R^{d_{cu}}$ with
    $\alpha_n(0,0)=0$, $\beta_m(0,0)=0$ and
$\|\alpha_n\|_{C^1}<\delta$, $\|\beta_n\|_{C^1}<\delta$.

Define $\lambda_n=\|A_n\|$, $\mu_n=\|B_n^{-1}\|^{-1}$ and
suppose that
$\lambda_n\ge\lambda_{min}$ and $\lambda_n/\mu_n\le\sigma$.
Set $\lambda_n'=(1+\gamma)(\lambda_n+\delta(1+\gamma))$,
$\mu_n'=\frac{\mu_n}{1+\gamma}-\delta$ and suppose that
$\lambda'_n<\nu_n<\mu'_n$ for all $n\in\Z$.

Then there exists a unique family of $d_s$-dimensional $C^1$ manifolds
 $Z_n=\{(x,\varphi_n(x)):x\in\R^{d_s}\}$,
 where
 $\varphi_n:\R^{d_s}\to\R^{d_{cu}}$ satisfies
 $\varphi_n(0,0)=0$, $D\varphi_n(0,0)=0$, and
 $\|D\varphi_n\|_{C^0}<\gamma$
 for all $n\in\Z$, and the following properties hold
for all $n\in\Z$:
 \begin{enumerate}
 \item $g_n(Z_n)=Z_{n+1}$,
 \item $\|g_n(q)\|\le \lambda_n'\|q\|$ for $q\in Z_n$,
 \item If
   $\|g_{n+k-1}\circ\cdots\circ g_n(q)\|\le C\nu_{n+k-1}\dots\nu_n\|q\|$ for
   all $k\ge0$ and some $C>0$, then $q\in Z_n$.
 \end{enumerate}
 If $\sup_n\lambda_n<1$, then the manifolds $Z_n$ are $C^r$.
  \end{lemma}

  \begin{proof}
    The only difference from \cite[Theorem 6.2.8, pp 242-257]{KH95} is
    that
    the rates $\lambda_n,\mu_n$ may depend on $n$.
      However, the ratios $\lambda_n/\mu_n$ 
      are controlled uniformly, and
    it is easy to check that the proof in pp 242-257 of \cite{KH95}
    is valid in this slightly more general setting with no
    change in the arguments.
  \end{proof}

\begin{remark}  The constraints on $\gamma$ and $\delta$ can be made explicit:
\[
	\gamma<\min\{1,\sigma^{-1/2}-1\}, \quad
\delta<\lambda_{min}\min\Bigl\{\frac{\sigma^{-1}-1}{\gamma+\gamma^{-1}+2}
,\frac{\sigma^{-1}-(1+\gamma)^2}{(2+\gamma)(1+\gamma)}\Bigr\}.
\]
\end{remark}

\begin{remark}
	In Lemma~\ref{lem:HP}, 
 there exists also a unique family of $d_{cu}$-dimensional $C^1$ manifolds
 $\tilde Z_n=\{(x,\psi_n(x)):x\in\R^{d_{cu}}\}$ satisfying analogous properties
 to the family $Z_n$.
 This leads to a family of center-unstable manifolds $\{W^{cu}_x,\,x\in\Lambda\}$ each of which is tangent at $x$ to $E^{cu}_x$.   These manifolds do not play a role in this paper.   (Unlike the case for stable manifolds, there is no useful notion of $W^{cu}_x$ for $x\not\in\Lambda$.)
 \end{remark}

Next, we verify the hypotheses of Lemma~\ref{lem:HP}.
Fix $x\in U_0$.
The sequence of diffeomorphisms $g_{x,n}:\R^d\to\R^d$ is defined for 
$n\ge0$.  For $n<0$, we set $g_{x,n}=g_{x,0}$. 
The diffeomorphisms $g_{x,n}$ have the structure in~\eqref{eq:gn}.
Take $\sigma=\hat\lambda\in(0,1)$ and
    $\lambda_{min}=\inf_{x\in U_0}\|DX_T\mid E^s_x\|>0$.
By Proposition~\ref{p:Dg}, the linear maps $A_n$, $B_n$ satisfy the constraints
$\lambda_{min}\le \lambda_n\le\sigma$ and $\lambda_n/\mu_n\le\sigma$.
Choose $\gamma$, $\delta>0$ so small that
$\sup_n\lambda_n'<1$ and $\sup_n\lambda_n'/\mu_n'<1$.
Choose $\nu_n\in(\lambda_n',\mu_n')$
such that $\nu=\sup_n\nu_n<1$.
Finally, shrink $\rho$ so that 
$\|\alpha_n\|_{C^1}<\delta$, $\|\beta_n\|_{C^1}<\delta$.

We have shown that the hypotheses of Lemma~\ref{lem:HP} are satisfied,
with $\nu_n\le \nu<1$ for all $n$.  
Let $Z_{x,n}$ denote the 
family of $d_s$-dimensional $C^r$ manifolds in Lemma~\ref{lem:HP} and
  define 
\[
W^s_x=Q_{x,0}(Z_{x,0} \cap  D_\rho).
\]
Repeating the construction for every $x\in U_0$, we obtain a family 
$\{W^s_x,\,x\in U_0\}$ of $d_s$-dimensional $C^r$ manifolds covering $U_0$.
We claim that this is the desired family of stable manifolds.

\begin{lemma}  \label{lem:Ws}
Let $x,y\in U_0$.
\begin{itemize}
	\item[(a)] If $d(x,y)<C_2^{-1}\rho$ and
	$y\in W^s_x$ then
$d(f^nx,f^ny)\le C_2^2\nu^n d(x,y)$ for all $n\ge0$.
\item[(b)] Let $C>0$.  If $d(x,y)<C_2^{-1}C^{-1}\rho$ and
$d(f^nx,f^ny)\le C\nu^n d(x,y)$ for all $n\ge0$, then $y\in W^s_x$.
\item[(c)] There exists $\eps>0$ such that if $d(x,y)<\eps$ and
$y\in W^s_x$ then
$fy\subset W^s_{fx}$.
\end{itemize}
\end{lemma}

\begin{proof}
Let 
\[
F_{x,n}=f_{x,n-1}\circ\cdots\circ f_{x,0}, \quad
G_{x,n}=g_{x,n-1}\circ\cdots\circ g_{x,0}.
\]
Note that if $F_{x,n}(q)\in D_\rho$ for all $0\le n\le N_0$, or
if $G_{x,n}(q)\in D_\rho$ for all $0\le n\le N_0$, then
$F_{x,n}(q)=G_{x,n}(q)$ for all $0\le n\le N_0$.

\vspace{1ex}
\noindent(a)
Let $y\in W^s_x$ with $d(x,y)<C_2^{-1}\rho$.
Then $q=Q_{x,0}^{-1}(y)\in Z_{x,0}$, so by Lemma~\ref{lem:HP}(1,2),
\[
	\|G_{x,n}(q)\|\le \nu^n\|q\|=\nu^n\|Q_{x,0}^{-1}(y)\|\le \nu^nC_2 d(x,y)<\rho,
\]
for all $n\ge0$.  
Now $f^n=Q_{x,n}\circ F_{x,n}\circ Q_{x,0}^{-1}$, so
\[
	f^ny=Q_{x,n}\circ F_{x,n}(q)
	=Q_{x,n}\circ G_{x,n}(q).
\]
Hence
\[
	d(f^nx,f^ny)=d(f^nx,Q_{x,n}\circ G_{x,n}(q))
	\le C_2\|G_{x,n}(q)\|\le C_2^2\nu^n d(x,y).
\]

\vspace{1ex}
\noindent(b)
Suppose that $d(x,y)<C_2^{-1}C^{-1}\rho$ and $d(f^nx,f^ny)\le C\nu^n d(x,y)$
for all $n\ge0$.  Let $q=Q_{x,0}^{-1}(y)$ so $d(x,y)\le C_2\|q\|$.
Now $F_{x,n}=Q_{x,n}^{-1}\circ f^n\circ Q_{x,0}$, so
\[
\|F_{x,n}(q)\|=
\|Q_{x,n}^{-1}\circ f^n(y)\|
\le C_2d(f^nx,f^ny)\le C_2C\nu^n d(x,y)<\rho.
\]
Hence
\[
\|G_{x,n}(q)\|=\|F_{x,n}(q)\|\le C_2C\nu^n d(x,y)\le C_2^2C\nu^n \|q\|.
\]
By Lemma~\ref{lem:HP}(3), $q\in Z_{x,0}\cap D_\rho$ and so $y=Q_{x,0}(q)\subset W^s_x$. 

\vspace{1ex}
\noindent(c)
Let $x'=fx$, $y'=fy$ and choose $E\ge1$ such that
$d(x,y)\le Ed(x',y')$ for all $x,y\in U_0$.

Suppose that $y\in W^s_x$ and $d(x,y)<C_2^{-5}E^{-1}\rho$.  Then certainly, $d(x,y)<C_2^{-1}\rho$, so by part~(a),
\[
d(f^nx',f^ny')=d(f^{n+1}x,f^{n+1}y)\le C_2^2\nu^{n+1}d(x,y)\le 
C_2^2E\nu^n d(x',y')=
C\nu^n d(x',y'),
\]
where $C=C_2^2E$.
Also, $d(x',y')\le C_2^2d(x,y)<C_2^{-3}E^{-1}\rho=C_2^{-1}C^{-1}\rho$, so the result follows from part (b).
\end{proof}

\begin{lemma}  \label{lem:Ws2}
	The $C^r$ embedded disks $W^s_x$ depend continuously on $x$ in the $C^0$ topology: there is a continuous map $\gamma:U_0\to {\rm Emb}^0(\D^{d_s},M)$ such that
$\gamma(x)(0)=x$ and $\gamma(x)(\D^{d_s})=W^s_x$.
Moreover, there exists $L\ge1$ such that $\Lip\gamma(x)\le L$ for all $x\in U_0$, where
$\Lip\gamma(x)=\sup_{u\neq u'}d(\gamma(x)(u),\gamma(x)(u'))/\|u-u'\|$.
\end{lemma}

\begin{proof}
	Fix $x\in U_0$ and recall that
	$W^s_x=Q_{x,0}(Z_{x,0}\cap D_\rho)$.
	For $y$ close to $x$, let $A_y=Q_{x,0}^{-1}(W^s_y)$.
Let $p_y=Q_{x,0}^{-1}(y)=Q_{x,0}^{-1}\circ Q_{y,0}(0)\in A_y$.
	In particular $A_x=Z_{x,0}\cap D_\rho$ and $p_x=0$.
Moreover, $y\mapsto p_y$ is continuous.

Now $T_{p_y}A_y= DQ_{x,0}^{-1}(y)T_yW^s_y=
DQ_{x,0}^{-1}(y)E^s_y $, so it follows from smoothness of
$Q_{x,0}$ and continuity of $E^s$ that $A_y$ can be viewed
as a graph over $\D^{d_s}\subset\R^{d_s}$ for $y$ close to~ $x$.  
In particular, $A_y=\{(u,\phi_y(u)):u\in \D^{d_s}\}$
where $\phi_y:\D^{d_s}\to\R^{d_{cu}}$, see Figure~\ref{fig:phi}.  Hence
$W^s_y=\{Q_{x,0}(u,\phi_y(u)):u\in\D^{d_s}\}$.  The family
of functions $\phi_y$ are $C^r$ with uniform Lipschitz
constant.  Since $p_y\in A_y$, there exists $u_y\in\D^{d_s}$
such that $p_y=(u_y,\phi_y(u_y))$.

  \begin{figure}[htpb]
    \centering
    \includegraphics[width=6cm]{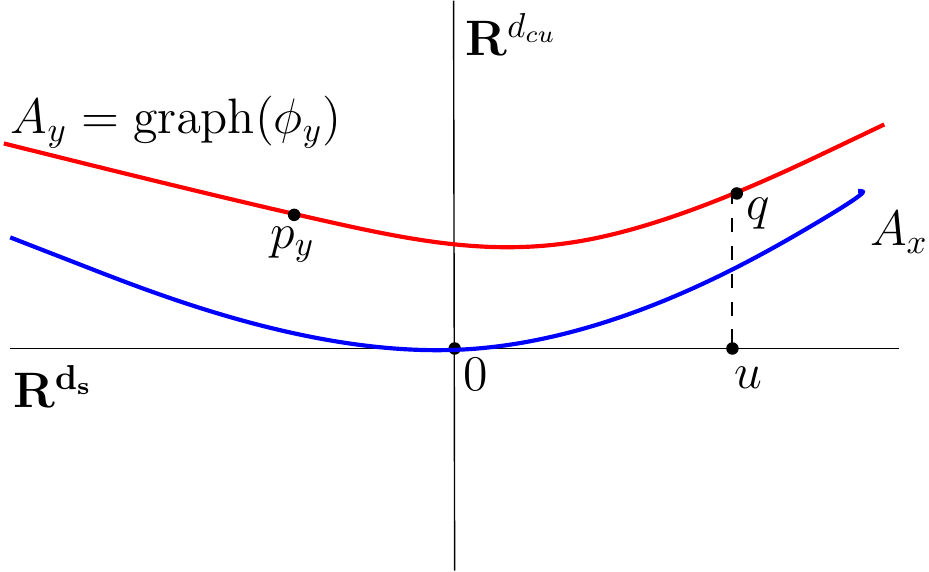}
    \caption{\label{fig:Ws-cont}$A_y$ as the graph of
      $\phi_y$ near $A_x$.}
\label{fig:phi}
  \end{figure}

Define the family of embeddings $\gamma:U_0\to{\rm Emb}^r(\D^{d_s},M)$ given by
\[
\gamma(y)(u)= Q_{x,0}(u,\phi_y(u)).
\] 
We claim that $y\mapsto \phi_y$ is continuous at $x$ in the $C^0$ topology, and hence the embedding $\gamma$ is continuous at $x$ in the $C^0$ topology.

It remains to verify the claim.
Suppose that $y_n\to x$.  By Arzel\`a-Ascoli,
we can pass to a further subsequence such that $\lim_{n\to\infty}\sup_{u\in\D^{d_s}}\|\phi_{y_n}(u)-\psi(u)\|=0$ for some
continuous function $\psi:\R^{d_s}\to\R^{d_{cu}}$.

Since $p_{y_n}\to0$, for $n$ large enough we have that $p_{y_n}\in D_{\frac12 C_2^{-5}\rho}$.  
Now fix $u\in\D^{d_s}$ (see Figure~\ref{fig:phi}).
Shrinking the disk $\D^{d_s}$, we can ensure that $q_n=(u,\phi_{y_n}(u))\in D_{\frac12 C_2^{-5}\rho}$ for $n$ sufficiently large.  
Hence 
\[
d(Q_{x,0}(q_n),y_n)\le d(Q_{x,0}(q_n),x)+d(x,y_n)
\le C_2^{-3}\rho\le C_2^{-1}\rho.
\]
By construction, $Q_{x,0}(q_n)\in W^s_{y_n}$, so by Lemma~\ref{lem:Ws}(a), 
\[
d(f^k\circ Q_{x,0}(q_n),f^ky_n)
\le C_2^2\nu^kd(Q_{x,0}(q_n),y_n)\quad\text{for all $k\ge0$}.
\]
Letting $n\to\infty$, we obtain that
\[
d(f^k\circ Q_{x,0}(u,\psi(u)),f^kx)
\le C_2^2\nu^kd(Q_{x,0}(u,\psi(u)),x)\quad\text{for all $k\ge0$}.
\]
By Lemma~\ref{lem:Ws}(b), 
$Q_{x,0}(u,\psi(u))\in W^s_x$ so $(u,\psi(u))\in A_x$.
It follows that $\psi(u)=\phi_x(u)$.  Hence all subsequential limits of $\phi_y$ (as $y\to x$) coincide with $\phi_x$ so $\lim_{y\to x}\phi_y=\phi_x$ in the
$C^0$ topology as required.
\end{proof}

\begin{lemma}\label{lem:Ws3}
The family of disks $\{W^s_x:x\in U_0\}$ defines a topological foliation.
\end{lemma}

\begin{proof}
Let $x\in U_0$ and choose an embedded $d_{cu}$-dimensional disk $Y\subset M$ containing $x$ and transverse to $W^s_x$.   By continuity of $E^s$, we can shrink $Y$ so that $Y$ is transverse to $W^s_y$ at $y$ for all $y\in Y$.
Let $\psi:\D^{cu}\to Y$ be a choice of embedding.

Now define $\chi:\D^s\times\D^{cu}\to U_0$ by setting 
$\chi(u,v)=\gamma(\psi(v))(u)$.
Note that $\chi$ maps horizontal lines $\{v={\rm const.}\}$ homeomorphically onto stable disks; see Figure~\ref{fig:topolchart}.
 \begin{figure}[htpb]
    \centering
    \includegraphics[width=8cm]{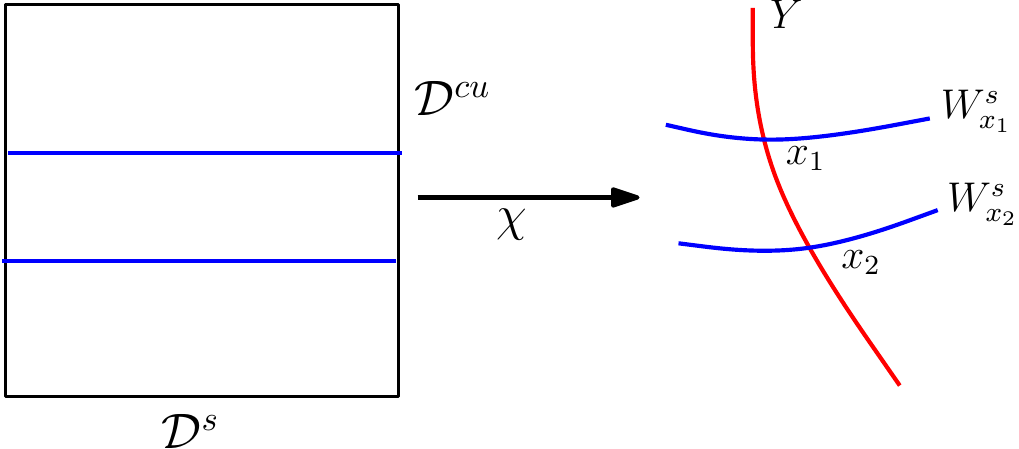}
    \caption{\label{fig:topolchart}Topological foliation chart}
  \end{figure}

By Lemma~\ref{lem:Ws2},
each of these embeddings is Lipschitz with uniform Lipschitz constant $L$.   Using this together with the continuity statement in Lemma~\ref{lem:Ws2}, 
\begin{align*}
d(\chi(u,v),\chi(u_0,v_0))& \le d(\gamma(\psi(v))(u),\gamma(\psi(v))(u_0))
+d(\gamma(\psi(v))(u_0),\gamma(\psi(v_0))(u_0)) \\
& \le L\|u-u_0\|+\|\gamma(\psi(v))-\gamma(\psi(v_0))\|_{C^0}\to0,
\end{align*}
as $(u,v)\to(u_0,v_0)$, establishing continuity of $\chi$.

Suppose that $\chi(u_1,v_1)=\chi(u_2,v_2)$ with common value $y\in U_0$.
Then $y\in W^s_{x_1}\cap W^s_{x_2}$ where $x_j=\psi(v_j)$.
We claim that $x_1=x_2$
with common value $\hat x$.  In particular $v_1=v_2$.
But now $\gamma(\hat x)(u_1)=\gamma(\hat x)(u_2)$ and so $u_1=u_2$.
It follows that $\chi$ is injective and hence is a homeomorphism onto a neighborhood of $x$ as required.

It remains to prove the claim.
Note that $W^s_{x_2}$ can be viewed as a graph over $W^s_{x_1}$.
Let $A=W^s_{x_1}\cap W^s_{x_2}$.  We show that $A$ is open and closed in $W^s_{x_1}$.   Since $y\in A$ and $W^s_{x_1}$ is connected, $A=W^s_{x_1}$ and in particular, $x_2=x_1$ as required.

It is clear that $A$ is closed in $W^s_{x_1}$.  To prove that $A$ is open, suppose that $z\in A$.  Since $W^s_{x_j}$ are tangent to $E^s_{x_j}$ with uniform Lipschitz constant, there exists $C>0$ such that
$d(x_1,x_2)\le Cd(z,x_j)$.
 Let $z'\in W^s_{x_1}$ be such that $d(z,z')\le (1/2C)d(x_1,x_2)$.   (This implies that $d(x_1,x_2)\le 2Cd(z',x_2)$.)  We must show that $z'\in A$.
Now
\begin{align*}
& d(f^nz',f^nx_2)  \le d(f^nz',f^nx_1)+d(f^nx_1,f^nz)+d(f^nz,f^nx_2)
 \\ & \le C_2^2\nu^n\{d(z',x_1)+d(x_1,z)+d(z,x_2)\}
\\ & \le C_2^2\nu^n\{d(z',x_2)+d(x_2,x_1)\;+\;d(x_1,x_2)+d(x_2,z')+d(z',z)\;+\;d(z,z')+d(z',x_2)\}
\\ & = C_2^2\nu^n\{3d(z',x_2)+2d(x_1,x_2)+2d(z,z')\}
\le C_2^2\nu^n\{3d(z',x_2)+4d(x_1,x_2)\} \\ &
\le (3+8C)C_2^2\nu^nd(z',x_2).
\end{align*}
We can arrange that $\chi$ takes values in $B_\eps(x)$ where $\eps$ is as small as required.  By Lemma~\ref{lem:Ws}(b), $z'\in W^s(x_2)$ and hence $z'\in A$ completing the proof.
\end{proof}

\begin{corollary} \label{cor:Ws}
There exists $\eps>0$ such that
$X_t(W^s_x\cap B_\eps(x))\subset W^s_{X_tx}$ for all $t\ge0$, $x\in U_0$.
\end{corollary}

\begin{proof}
Choose $n_0\ge1$ such that $C_2^2\nu^{n_0}<1$.  Shrinking $\eps$, it follows from
Lemma~\ref{lem:Ws}(a,c) that 
$f^{n_0}(W^s_x\cap B_\eps(x))\subset W^s_{f^{n_0}x}\cap B_\eps(f^{n_0}x)$ and inductively that
$f^{kn_0}(W^s_x\cap B_\eps(x))\subset W^s_{f^{kn_0}x}\cap B_\eps(f^{kn_0}x)$ for all $k\ge0$.

Next choose $C\ge1$ such that $d(X_rx,X_ry)\le Cd(x,y)$ for all $x,y\in U_0$, $r\in[-n_0T,n_0T]$.
Suppose that $y\in W^s_x$ and let $x'=X_rx$, $y'=X_ry$.  By Lemma~\ref{lem:Ws}(a),
for $y$ sufficiently close to $x$,
\[
d(f^nx',f^ny')=d(X_rf^nx,X_rf^ny)\le Cd(f^nx,f^ny)\le CC_2^2\nu^nd(x,y)\le C^2C_2^2\nu^nd(x',y'),
\]
for all $n\ge0$.
By Lemma~\ref{lem:Ws}(b), $X_ry\in W^s_{X_rx}$ for $y$ sufficiently close to $x$.
Hence there exists $\eps>0$ such that $X_r(W^s_x\cap B_\eps(x))\subset W^s_{X_rx}$ for all $r\in[0,n_0T]$, $x\in U_0$.

The result for general $t$ follows by writing
$t=kn_0T+r$ where $k\ge0$, $r\in[0,n_0T)$.~
\end{proof}

Recall that $f=X_T$.  Choose $C$ such that
$\sup_{r\in[0,T]}d(X_rx,X_ry)\le Cd(x,y)$ for all $x,y\in U$.
Write $t=nT+r$, $n\ge0$, $r\in[0,T)$.
By Lemma~\ref{lem:Ws}(a),
if $d(x,y)<C_2^{-1}\rho$ and $y\in W^s_x$, then 
\[
d(X_tx,X_ty)=d(X_{nT+r}x,X_{nT+r}y)\le C_2^2C\nu^nd(x,y)
\le C'\tilde\nu^t d(x,y),
\]
where $C'=C_2^2C\nu^{-1}$ and $\tilde\nu=\nu^{1/T}$.

Passing to an adapted metric, we can arrange that there are constants $\eps>0$, $\nu\in(0,1)$ such that if $d(x,y)<\eps$
and $y\in W^s_x$, then $d(X_tx,X_ty)\le \nu^t d(x,y)$ for all $t\ge0$.
From now on, we write $W^s_x$ instead of $W^s_x\cap B_\eps(x)$.
With this notation, Corollary~\ref{cor:Ws} states that 
$X_t(W^s_x)\subset W^s_{X_tx}$ for all $x\in U_0$, $t\ge0$.

This completes the proof of Theorem~\ref{th:Ws}.

\subsection{Regularity of the stable foliation}
\label{sec:reg}

In this subsection, we prove results on the regularity of the stable foliation
$\{W^s_x\}$ in a neighborhood of $\Lambda$ under an appropriate bunching condition.
We follow \cite[Theorem 6.5]{HP70}, adapting and applying
the results of \cite{HPS77} in our setting. 

We continue to suppose that $X_t$ is the flow generated by a $C^r$ vector field $G$ where $r\ge1$.   
Let $q\in[0,r]$.  We suppose that there exists $t>0$
such that the following bunching condition holds:
\begin{align}\label{eq:bunching}
  \|DX_t\mid E^s_x\|\cdot \|DX_{-t}\mid E^{cu}_{X_tx}\|
  \cdot\|DX_t\mid E^{cu}_x\|^q <1
  \quad\text{for all $x\in \Lambda$}.
\end{align}

Let $T_{U_0}M=E^s\oplus E^{cu}$ be a continuous extension of the underlying splitting with $E^s$ invariant as in Proposition~\ref{p:hatEs}.
Choose $t$ as in~\eqref{eq:bunching} and let $f=X_t$.
Increasing $t$ and shrinking $U_0$ if necessary, we can ensure that
\begin{align}\label{eq:bunching2}
  \|Df\mid E^s_x\|\cdot \|Df^{-1}\mid E^{cu}_{fx}\| \le
  \|Df\mid E^s_x\|\cdot \|Df^{-1}\mid E^{cu}_{fx}\| 
\cdot\|Df\mid T_xM\|^q <1,
\end{align}
for all $x\in U_0$.

Let $T_{U_0}M=F^s\oplus F^{cu}$ be a $C^r$ approximation of the splitting
$T_{U_0}M=E^s\oplus E^{cu}$. 
For each $x\in U_0$, let $L(F^s_x,F^{cu}_x)$ denote the space of linear maps from
$F^s_x$ to $F^{cu}_x$, and let $\D_x$ denote the unit disk in
$L(F^s_x,F^{cu}_x)$.
 Define the corresponding disk bundle
$\cD_0=\{\D_x,\,x\in U_0\}$.

Let  $U_1=f(U_0)\subset U_0$ and set
$\cD_1=\{\D_x,\,x\in U_1\}$.
Let $h=f^{-1}|_{U_1}:U_1\to U_0$.  Since $h(U_1)=U_0\supset U_1$, the $C^r$ diffeomorphism $h$ is {\em overflowing} in the sense of~\cite[p.~30]{HPS77}.

Represent $Dh(x):T_xM\to T_{hx}M$ using the splitting $F^s\oplus F^{cu}$ by writing
\begin{align*}
  Dh(x)=\left(\begin{matrix}
    A_x & B_x\\C_x & D_x
  \end{matrix}\right): F_x^s\times F^{cu}_x\to F^s_{hx}\times
  F^{cu}_{hx}, \quad x\in U_1.
\end{align*}
We define the graph transform $\Gamma:\cD_1\to\cD_0$,
\begin{align*}
  \Gamma_x(\ell)=(C_x+D_x\ell)(A_x+B_x\ell)^{-1},
  \quad \ell\in\D_x,\, x\in U_1.
\end{align*}

\begin{lemma}   \label{lem:reg}
The neighborhood $U_0$ of $\Lambda$ and the $C^r$ splitting $F^s\oplus F^{cu}$ can be chosen so that
$\Gamma:\cD_1\to\cD_0$ is well-defined and 
$\Lip(\Gamma_x)\cdot\|Dh^{-1}|T_{hx}M\|^q<1$ for all $x\in U_1$.
\end{lemma}

\begin{proof}
By~\eqref{eq:bunching2}, we can choose $\lambda_x\in(0,1)$ such that
\[
  \|Df\mid E^s_x\|\cdot \|Df^{-1}\mid E^{cu}_{fx}\| <
\lambda_x \quad\text{and}\quad
\lambda_x\, \|Df\mid T_xM\|^q <1\quad\text{for all $x\in U_0$.}
\]
Since $f$ is $C^1$ and $\overline{U_0}$ is compact, there exists
$\delta\in(0,1)$ such that
$(\lambda_{hx}+2\delta)(1-\delta)^{-2}<1$ and
\begin{align} \label{eq:delta}
(\lambda_{hx}+2\delta)(1-\delta)^{-2} \|Dh^{-1}\mid T_{hx}M\|^q <1\quad\text{for all $x\in U_1$.}
\end{align}

Since $F^s$ is close to the $Df$-invariant contracting bundle $E^s$, we
can arrange that 
$\|C_x\|\le1$ and $\|A_x^{-1}\|\le 1$ for all $x\in U_1$.
Also, $F^{cu}$ is close to $E^{cu}$ which is invariant when restricted to $\Lambda$ so we can arrange that $\|B_x\|<\delta$.
Moreover, $A_x^{-1}$ is close to $Df\mid E^s_{hx}$ and
$D_x$ is close to $Df^{-1}\mid E^{cu}_x$ so we can ensure that
 $\|A_x^{-1}\|\|D_x\|\le\lambda_{hx}$
for all $x\in U_1$.

Let $\ell,\ell'\in \D_x$.
Note that $\|A_x^{-1}B_x\ell\|\le \delta$, so 
$\|(I+A_x^{-1}B_x\ell)^{-1}\|\le (1-\delta)^{-1}$.
Similarly, $\|(I+A_x^{-1}B_x\ell')^{-1}\|\le (1-\delta)^{-1}$.
It follows that
\begin{align*}
   \|(A_x+B_x\ell)^{-1}-  (A_x+B_x\ell' & )^{-1}  \| 
 =\|(A_x+B_x\ell)^{-1}(B_x(\ell'-\ell))(A_x+B_x\ell')^{-1}\| \\
  & \le \|A_x^{-1}\|^2\delta(1-\delta)^{-2}\|\ell'-\ell\|
   \le \|A_x^{-1}\|\delta(1-\delta)^{-2}\|\ell'-\ell\|.
\end{align*}
Hence
\begin{align*}
\|\Gamma_x(\ell)-\Gamma_x & (\ell')\|  
\le \|D_x(\ell-\ell')\| \|(A_x+B_x\ell)^{-1}\|
\\ &\qquad\qquad \qquad +\|(C_x+D_x\ell')\| \|(A_x+B_x\ell)^{-1}-(A_x+B_x\ell')^{-1}\| \\
& \le \|A_x\|^{-1}\|D_x\|(1-\delta)^{-1}\|\ell-\ell'\|
+ (1+\|D_x\|)\|A_x^{-1}\|\delta(1-\delta)^{-2}\|\ell-\ell'\| \\
& \le \lambda_{hx}(1-\delta)^{-1}\|\ell-\ell'\|
+ 2\delta(1-\delta)^{-2}\|\ell-\ell'\|,
\end{align*}
and so
$\Lip(\Gamma_x)\le (\lambda_{hx}+2\delta)(1-\delta)^{-2}$
for all $x\in U_1$.   In particular, $\Lip(\Gamma_x)<1$ so $\Gamma_x(\D_x)\subset \D_{hx}$, and hence $\Gamma$ is well-defined.
The result follows from this estimate combined with~\eqref{eq:delta}.
\end{proof}

\begin{theorem} \label{th:reg}
Let $q\in[0, [r]\,]$.
If condition~\eqref{eq:bunching} holds for some $t>0$,
then the bundle $E^s$ is $C^q$ over $U_1$.
That is, the map $x\mapsto E^s_x$ is a $C^q$ map from $U_1$ to $\cD_1$.
\end{theorem}

\begin{proof}
Recall that we can regard $E^s_x$ as the graph of an element $\omega\in L(F^s_x,F^{cu}_x)$ with $\|\omega\|$ as close to zero as desired.  In particular, $\|\omega\|\le1$, and hence
$E^s$ is identified with a continuous $Df$-invariant section of $\cD_1$.

Note that
$Dh(x)\graph(\ell)=\graph(\Gamma_x(\ell))$ for $\ell\in \D_x$. 
Since $h=df^{-1}$, it follows that
$E^s:U_1\to\cD_1$ is a continuous $\Gamma$-invariant section.

By Lemma~\ref{lem:reg}, the graph transform $\Gamma:\cD_1\to\cD_0$ defines 
a fiber contraction over the overflowing diffeomorphism $h:U_1\to U_0$,
and this fiber contraction is {\em $q$-sharp} in the terminology of~\cite{HPS77}.   
When $q$ is an integer,
we have verified the hypotheses of the ``$C^r$ section theorem''~\cite[Theorem~3.5]{HPS77} (with $q$ playing the role of $r$, and vector bundles replaced by disk bundles as in~\cite[Remark, p.~36]{HPS77}). 
It follows from~\cite[Theorem~3.5]{HPS77} that $E^s:U_1\to\cD_1$ is the unique continuous $\Gamma$-invariant section and moreover that this section is $C^q$.

This completes the proof in the case that $q$ is an integer.
The general case follows from~\cite[Remark~2, p.~38]{HPS77}.
\end{proof}

\begin{remark} \label{rmk:reg}
(a)  It is immediate from domination~\eqref{eq:domination} that condition~\eqref{eq:bunching} holds with $q=0$.   By smoothness of the flow and compactness of $\Lambda$, condition~\eqref{eq:bunching} holds for some $q>0$ and hence the stable bundle $E^s$ is at least H\"older over $U_1$.

\vspace{1ex}
\noindent(b) When $q\ge1$ in Theorem~\ref{th:reg}, it follows by Frobenius's Theorem (see for example~\cite[pp.\ 93--95]{AbrahamMarsden}) 
that the family of stable manifolds $\{W^s_x\}$ obtained in
Theorem~\ref{th:Ws} forms a $C^q$ foliation of
$U_1$ in the sense that the foliation charts are $C^q$.  Moreover the holonomy maps along the stable leaves are $C^q$ smooth.
(See~\cite[Section~6]{PSW97} for more details.)
For $q\in(0,1)$, it remains true that the holonomy maps are $C^q$~\cite{PSW97}.
\end{remark}

\section{Strong dissipativity}
\label{sec:SD}

From now on, our results are specialised to flows.
In this section, we define strong dissipativity for sectional hyperbolic attractors.  This is a verifiable condition for smoothness of stable foliations, extending~\cite{AMV15} who proved strong dissipativity and hence smoothness of the stable foliation for the classical Lorenz attractor.  We recover the result of~\cite{AMV15} and moreover obtain an estimate for the smoothness.

Recall that $d_s=\dim E^s_x$.
Given $A=\{a_{ij}\}\in\R^{d\times d}$,
let $\|A\|_2=(\sum_{ij}a_{ij}^2)^{1/2}$.

\begin{definition} \label{def:SD}
Let $q>1/d_s$.
A partially hyperbolic attractor $\Lambda$ is {\em $q$-strongly dissipative} if
\begin{itemize}
\item[(a)] 
For every equilibrium $\sigma\in\Lambda$ (if any), the
  eigenvalues $\lambda_j$ of $DG(\sigma)$, ordered so that $\Re\lambda_1\le\Re\lambda_2\le \dots\le\Re\lambda_d$,
  satisfy
  $\Re(\lambda_1-\lambda_{d_s+1}+q\lambda_d)<0$.
\item[(b)] $\sup_{x\in\Lambda}\bigl\{\diver G(x)+(d_sq-1)\|(DG)(x)\|_2\bigr\}<0$.
\end{itemize}
\end{definition}

\begin{theorem} \label{th:SD} Let $\Lambda$ be a sectional hyperbolic attractor.
Suppose that $\Lambda$ is $q$-strongly dissipative for some $q\in(1/d_s,[r]\,]$.
Then there exists a neighborhood $U_0$ of $\Lambda$ such that the stable manifolds $\{W^s_x,\,x\in U_0\}$ define a $C^q$ foliation of $U_0$.
\end{theorem}

\begin{proof}
For each $t\in\R$, we define $\eta_t:\Lambda\to\R$,
\[
\eta_t(x)=\log\big\{\|DX_t|E^s_x\|\cdot
\|DX_{-t}|E^{cu}_{X_tx}\|\cdot
\|DX_t|E^{cu}_x\|^q\big\}.
\]
Note that $\{\eta_t,\;t\in\R\}$ is a continuous family of
continuous functions each of which is subadditive, that is,
$\eta_{s+t}(x)\le \eta_s(x)+\eta_t(X_sx)$.

Let $\cM$ denote the set of flow-invariant ergodic probability measures on $\Lambda$.  We claim that for each $m\in\cM$, the limit
$\lim_{t\to\infty}t^{-1} \eta_t(x)$ exists and is negative for $m$-almost every $x\in\Lambda$.
It then follows from~\cite[Proposition~3.4]{arbieto2010}
that there exists constants $C,\beta>0$ such that $\exp\eta_t(x)\le Ce^{-\beta t}$ for all $t>0$, $x\in\Lambda$.
In particular, for $t$ sufficiently large,
$\exp\eta_t(x)<1$ for all $x\in\Lambda$.  Hence condition~\eqref{eq:bunching} is satisfied for such 
$t$ and the result follows from Theorem~\ref{th:reg} and
Remark~\ref{rmk:reg}.

It remains to verify the claim.
For each $m\in\cM$, we label the Lyapunov exponents 
\[
\chi_1(m)\le \chi_2(m)\le\dots\le\chi_d(m).
\]
Since $\Lambda$ is partially hyperbolic, the Lyapunov exponents $\chi_j(m)$,
$j=1,\dots,d_s$ are associated with $E^s$ and are negative,
while the remaining exponents are associated with $E^{cu}$.

For $m$-a.e. $x\in\Lambda$ we have
\begin{align*} 
& \lim_{t\to\infty}t^{-1} \log \|DX_t|E^s_x\|=\chi_1(m), \quad
 \lim_{t\to\infty}t^{-1} \log \|DX_{-t}|E^{cu}_{X_tx}\|=-\chi_{d_s+1}(m), \\
& \lim_{t\to\infty}t^{-1} \log \|DX_t|E^{cu}_x\|=
 \lim_{t\to\infty}t^{-1} \log \|DX_t\mid T_xM\|=
\chi_d(m).
\end{align*}
Hence, $m$-almost everywhere,
\begin{align*} 
\lim_{t\to\infty}t^{-1}\eta_t(x)=\chi_1(m)-\chi_{d_s+1}(m)+q\chi_d(m).
\end{align*}

If $m$ is a Dirac delta at an equilibrium $\sigma\in \Lambda$,
then $\chi_j(m)=\Re\lambda_j$ for $j=1,\dots,d$, where $\lambda_j$ are the eigenvalues of $DG(\sigma)$.
Hence, it is immediate from Definition~\ref{def:SD}(a) that
$\lim_{t\to\infty}t^{-1}\eta_t(\sigma)<0$.

If $m$ is not supported on an equilibrium, then there is a zero Lyapunov exponent in the flow direction.   Sectional expansion ensures that $\chi_{d_s+1}(m)=0$ and  $\chi_j(m)>0$ for $j=d_s+2,\dots,d$.
Hence, $m$-almost everywhere,
\begin{align*}
\lim_{t\to\infty}t^{-1}\eta_t(x) & \textstyle =\chi_1(m)+q\chi_d(m)
\le d_s^{-1}\sum_{j=1}^{d_s}\chi_j(m)+q\chi_d(m)
\\ & \textstyle = d_s^{-1}\bigl(\sum_{j=1}^{d_s}\chi_j(m)+d_sq\chi_d(m)\bigr)
\le d_s^{-1}\bigl(\sum_{j=1}^d\chi_j(m)+(d_sq-1)\chi_d(m)\bigr)
\\ & \textstyle = d_s^{-1}\lim_{t\to\infty}t^{-1} \bigl(\log |\det DX_t(x)|+(d_sq-1)
\log\|DX_t(x)\|\bigr)
\\ & \textstyle \le  d_s^{-1}\lim_{t\to\infty}t^{-1}\int_0^t \bigl(\diver DG(X_sx)+(d_sq-1) \|DG(X_sx)\|_2\bigr)\,ds
\\ & \textstyle \le  d_s^{-1}\sup_{x\in\Lambda}\big\{\diver DG(x)+(d_sq-1) \|DG(x)\|_2\bigr\}.
\end{align*}
By Definition~\ref{def:SD}(b), we again have that
$\lim_{t\to\infty}t^{-1} \eta_t(x)<0$ for $m$-almost every $x\in\Lambda$.
This completes the proof of the claim.
\end{proof}

\begin{remark}  If $\sup_\Lambda\diver G<0$, then condition (b) holds for
$q=d_s^{-1}+\eps$ for $\eps$ sufficiently small.  When $\dim M=3$, we have $d_s=1$ and hence we recover the result in~\cite[Lemma~2.2]{AMV15}.
For the classical Lorenz equations~\cite{Lorenz63}, we have
\[
\diver G\equiv -\textstyle{\frac{41}{3}}, \quad
\lambda_1\approx -22.83, \quad
\lambda_2=-\textstyle{\frac83}, \quad \lambda_3\approx 11.83,
\]
so the Lorenz attractor is $(1+\eps)$-strongly dissipative for 
$\eps>0$ sufficiently small.
Hence, the stable foliation is $C^{1+\eps}$ for the classical Lorenz attractor.
\end{remark}

In fact, we have:

\begin{corollary} 
The stable foliation for the classical Lorenz attractor is at least
$C^{1.278}$.
\end{corollary}

\begin{proof}  
By definition, $q$-strong dissipativity holds for any $q<\min\{q_1,q_2\}$ where
\begin{align*}
q_1 & =\frac{\lambda_2-\lambda_1}{\lambda_3}\approx 1.704, \quad
q_2  = 1-\frac{\diver G}{\sup\|DG\|_2} = 1+\frac{41}{3}\frac{1}{\sup\|DG\|_2}.
\end{align*}
Now 
\[
\textstyle \|DG(x)\|_2^2=201+\frac{64}{9}+2x_1^2+x_2^2+(x_3-28)^2
\le 208.12+V,
\]
where 
\[
V=2x_1^2+x_2^2+(x_3-28)^2.
\]
We claim that $V \le 2197$.  Then $q\ge 1.278$ as required.

In verifying the claim, we use the standard notation
$\sigma=10$, $b=\frac83$, $r=28$ for the parameters in 
the Lorenz equations.
Various authors~\cite{DoeringGibbon95,GiacominiNeukirch,Dyer01}
have established the estimate
\begin{align} \label{eq-various}
x_2^2+(x_3-r)^2\le R^2, \quad R^2= \frac{b^2r^2}{4(b-1)}=\frac{12544}{15}\le 836.27\,.
\end{align}
It follows that
\begin{align} \label{eq-square}
x_2^2+x_3^2\le (r+R)^2
\le  3239.7\,.
\end{align}

Next, by~\cite[Example~5]{Dyer01} (see also~\cite{GiacominiNeukirch}),
$ax_1^2\le (x_2^2+x_3^2)$ where $a>0$ is the largest root of
\[
(10a-28)^2+a(20-\lambda)(2-\lambda), \qquad\lambda=11.
\]
This yields
$a\ge 4.7644$, so using~\eqref{eq-square},
$x_1^2\le a^{-1}(x_2^2+x_3^2)\le  680$.
Combined with~\eqref{eq-various},
\[
V=2x_1^2+x_2^2+(x_3-28)^2 \le 1360+837=2197,
\]
proving the claim.
\end{proof}

\begin{remark}
It might seem that we could use~\cite[equation~(22)]{DoeringGibbon95}
to get a better estimate on $x_3$ and hence $V$.
Unfortunately this equation implies
\[
\textstyle \big(x_3-\frac12(r+\sigma)\big)^2\le\frac14(r+\sigma)^2,
\]
so $x_3\le r+\sigma=38$ which is clearly incorrect.
\end{remark}



\def\cprime{$'$}


\begin{thebibliography}{10}

\bibitem{AbrahamMarsden}
R. Abraham and J. Marsden.
\newblock {\em Foundations of Mechanics}, 2nd ed., Benjamin/Cummings, Reading, Mass., 1978.

 \bibitem{AMsub}
 V.~Araujo and I.~Melbourne.
\newblock  Exponential decay of correlations for nonuniformly hyperbolic flows with a $C^{1+\alpha}$ stable foliation, including the classical Lorenz attractor.
 \newblock {\em Annales Henri Poincar\'e} {\bf 17} (2016) 2975--3004.



 \bibitem{AMV15}
 V.~Araujo, I.~Melbourne and P.~Varandas.
 \newblock Rapid mixing for the Lorenz attractor and statistical limit laws for
   their time-1 maps.
 \newblock {\em Commun. Math. Phys.} {\bf 340} (2015) 901--938.
 
\bibitem{AraPac2010}
V.~Ara{\'u}jo and M.~J. Pacifico.
\newblock {\em Three-dimensional flows}. Volume~53 of {\em Ergebnisse der
  Mathematik und ihrer Grenzgebiete. 3. Folge. A Series of Modern Surveys in
  Mathematics [Results in Mathematics and Related Areas. 3rd Series. A Series
  of Modern Surveys in Mathematics]}.
\newblock Springer, Heidelberg, 2010.

 \bibitem{arbieto2010}
 A.~Arbieto.
 \newblock Sectional {L}yapunov exponents.
 \newblock {\em Proc. Amer. Math. Soc.} {\bf 138} (2010) 3171--3178.
 
 \bibitem{DoeringGibbon95}
C.~Doering and J. Gibbon.
\newblock On the shape and dimension of the Lorenz attractor.
\newblock {\em Dyn. Stab. Systems} {\bf  10} (1995) 255--268.

\bibitem{GiacominiNeukirch}
 H. Giacomini and S. Neukirch.
 \newblock Integrals of motion and the shape of the attractor for the Lorenz model.
\newblock {\em Phys.\ Lett.\ A} {\bf 227} (1997) 309--318.

 \bibitem{Goum07}
 N.~Gourmelon.
 \newblock Adapted metrics for dominated splittings.
 \newblock {\em Ergodic Theory Dynam. Systems} {\bf 27} (2007) 1839--1849.
 
 \bibitem{HPS77}
 M.~Hirsch, C.~Pugh and M.~Shub.
 \newblock {\em {Invariant manifolds}}, volume {583} of {\em {Lect. Notes in
   Math.}}
 \newblock {Springer Verlag}, {New York}, {1977}.
 
 \bibitem{HP70}
 M.~W. Hirsch and C.~C. Pugh.
 \newblock Stable manifolds and hyperbolic sets.
 \newblock In {\em Global {A}nalysis ({P}roc. {S}ympos. {P}ure {M}ath., {V}ol.
   {XIV}, {B}erkeley, {C}alif., 1968)}, pages 133--163. Amer. Math. Soc.,
   Providence, R.I., 1970.
 
 \bibitem{Homburg96}
 A.J. Homburg.
 \newblock Global aspects of homoclinic bifurcations of vector fields.
 \newblock {\emph Memoirs of the Amer. Math. Soc.} {\bf 578}, 1996.
 
 \bibitem{KH95}
 A.~Katok and B.~Hasselblatt.
 \newblock {\em {Introduction to the modern theory of dynamical systems}},
   volume~{54} of {\em {Encyclopeadia Appl. Math.}}
 \newblock {Cambridge University Press}, {Cambridge}, {1995}.
 
 \bibitem{Lorenz63}
 E.~D. Lorenz.
 \newblock Deterministic nonperiodic flow.
 \newblock \emph{J. Atmosph. Sci.} {\bf 20} (1963) 130--141.
 
 \bibitem{MPP04}
 C.~A. Morales, M.~J. Pacifico and E.~R. Pujals.
 \newblock {Robust transitive singular sets for 3-flows are partially hyperbolic
   attractors or repellers}.
 \newblock {\em {Ann. of Math. (2)}} {\bf 160} (2004) 375--432.
 
 \bibitem{PSW97}
 C.~Pugh, M.~Shub and A.~Wilkinson.
 \newblock {H{\"o}lder foliations}.
 \newblock {\em {Duke Math. J.}} {\bf 86} (1997) 517--546.
 
 \bibitem{Robinson92}
 C. Robinson.
 Homoclinic bifurcation to a transitive attractor of Lorenz type, II.
  \emph{SIAM J. Math. Anal.} {\bf 23} (1992) 1255--1268.
 
   \bibitem{Rychlik90}
 M.~R.~Rychlik.
 \newblock  Lorenz attractors through \v{S}il\cprime nikov-type bifurcation. I.
 \newblock         \emph{Ergodic Theory Dynam. Systems} {\bf 10} (1990) 
793--821.
 

\bibitem{Dyer01}
P.~Swinnerton-Dyer.
\newblock Bounds for trajectories of the Lorenz equations:
an illustration of how to choose Liapunov functions.
\newblock  \emph{Phys. Lett. A} {\bf 281} (2001) 161--167.

 \bibitem{Tucker}
   W.~Tucker.
   \newblock A rigorous ODE solver and Smale's 14th problem.
   \newblock {\em Found. Comput. Math.} {\bf 2} (2002) 53--117.
 
  \end{thebibliography}
\end{document}